\documentclass[12pt]{article}
\usepackage{amsmath,amsfonts,amssymb,amsthm,amscd}
\title{  \LARGE  Remarks on convergence of Morley sequences \\ %On the Fremlin--Talagrand stability  and Keisler measures: \\  On  $NIP$ and differences of semicontinuous functions:~I
}

\author{Karim Khanaki\thanks{Partially supported by IPM grant 1400030118}\\Arak University of Technology \\ \\ {\em Dedicated to the memory of my first teacher Fatemeh Mardani}}

\newtheorem{Theorem}{Theorem}[section]
\newtheorem{Proposition}[Theorem]{Proposition}
\newtheorem{Definition}[Theorem]{Definition}
\newtheorem{Remark}[Theorem]{Remark}
\newtheorem{Lemma}[Theorem]{Lemma}
\newtheorem{Corollary}[Theorem]{Corollary}
\newtheorem{Fact}[Theorem]{Fact}

\newtheorem{Convention}[Theorem]{Convention}

\def\dotminus{\mathbin{\ooalign{\hss\raise1ex\hbox{.}\hss\cr
  \mathsurround=0pt$-$}}}

\begin{document}
\maketitle

\begin{abstract}   We refine results of  Gannon \cite[Theorem~4.7]{Gannon sequential} and Simon \cite[Lemma~2.8]{S-invariant} on  convergence of Morley sequences. We then introduce the notion of {\em eventual $NIP$}, as a property of a model, and prove a variant of \cite[Corollary~2.2]{KP}. Finally, we give  new characterizations of generically stable types (for countable theories) and reinforce the main result of Pillay \cite{Pillay} on  the model-theoretic meaning of   Grothendieck's double limit theorem.
\end{abstract}

\section{Introduction} \label{1} 
Pillay and Tanović \cite{Pillay-Tanovic} introduced the notion of {\em generically stable type}, for {\em arbitrary} theories, as an abstraction of the crucial properties of definable types in stable theories. All invariant types in $NIP$ theories and all generically stable types in arbitrary theories  share an important phenomenon: {\em convergence of Morley sequences.} 
Using this phenomenon/property, although it is not explicitly mentioned, Simon \cite{S-invariant} proved the following interesting result:

\medskip\noindent {\bf Simon's Lemma\footnote{\cite[Lemma~2.8]{S-invariant}. In this article, when we refer to Simon's Lemma,	we mean this result.}:} Let $T$ be a countable $NIP$ theory and $M$ a countable model of $T$. Suppose that $p(x)\in S(\cal U)$ is  finitely satisfiable in $M$. Then  there is a sequence $(c_i)$ in $M$ such that $\lim tp(c_i/{\cal U})=p$.

The present paper 
aims to focus on convergence of Morley sequences.
The core of our observations/proofs here is that the convergence of tuples/types depends on a certain type of formulas, namely {\em symmetric formulas}. We show that a sequence of types converges if and only if there are some  symmetric  formulas that are not true in the sequence.

\medskip\noindent
On the other hand, the origin of Simon's lemma is related to the following crucial theorem in functional analysis due to Bourgain, Fremlin, and Talagrand \cite[Thm.~3F]{BFT}:

\medskip\noindent {\bf BFT Theorem:} Let $X$ be a Polish space. Then the space $B_1(X)$ of all Baire~1 functions on $X$ is an angelic space with the topology of  pointwise convergence.

\medskip\noindent
In \cite[Appendix~A]{K-GC}, it is shown that {\bf complete} types (not just   $\phi$-types) can be coded by  suitable functions, and a refinement of Simon's lemma is given using the BFT theorem. Recall that every point in the closure of a relatively
compact set of an angelic space is the limit of a {\bf sequence} of its points, and  relatively
compact sets of  $B_1(X)$ possess a property similar to $NIP$.
(Cf. \cite{KP}.)
In this paper we aim  to give a model theoretic version of the  Bourgain-Fremlin-Talagrand result in the terms of convergent Morley sequences.\footnote{Although we only use one analytical/combinatorial result (Fact~\ref{Rosenthal lemma}), we will explain that it is not even needed and that all arguments in this paper are model-theoretic. Cf. Remark~\ref{A=any}(iv) below.}  More precisely, we prove that: 

\medskip\noindent {\bf Theorem A:}
	Let $T$ be a countable theory and $M$ a countable model of  $T$.\footnote{We can consider countable fragments of (uncountable) theories, however, to make the proofs more readable, we assume that the theory is countable.} 
	\newline
	(1) Suppose that  $p(x)\in S(\cal U)$  is finitely satisfiable in $M$ and there is a Morley sequence $(d_i)$ of $p$ over $M$ such that $(tp(d_i/{\cal U}):i<\omega)$ converges. Then there is a sequence $(c_i)\in M$ such that  $\lim tp(c_i/{\cal U})=p$.
	\newline	
	(2) Furthermore, the following are equivalent:
	\newline
	(i) $M$ is eventually $NIP$ (as in Definition~\ref{eventual NIP}). 
	\newline
	(ii) For any $p(x)\in S(\cal U)$ which  is  finitely satisfiable in $M$, there is a sequence $(c_i)\in M$ such that the sequence $(tp(c_i/{\cal U}):i<\omega)$ $DBSC$-converges to  $p$ (as in Definition~\ref{DBSC converge}).
	\newline
	(iii) For any $p(x)\in S(\cal U)$ which  is finitely satisfiable in $M$, there is a Morley sequence $(d_i)$ of $p$ over $M$ such that $(tp(d_i/{\cal U}):i<\omega)$ converges.
 
\medskip
Recall that Morley sequences in $NIP$ theories are convergent (cf. Definition~\ref{covergence}).\footnote{This is a consequence of indiscernibility of Morley sequences and countability of theory. If not, one can find a Morley sequence $(a_i)$ and a formula $\phi(x)$ such that $\models\phi(a_i)$ iff $i$ is even, a contradiction.} Therefore, as the theory $T$ in   Theorem A is arbitrary,  the equivalences (i)--(iii) of Theorem A(2)  refine  Simon's lemma above. On the other hand,  a result of Gannon \cite[Thm.~4.7]{Gannon sequential}  asserts that:  

\medskip\noindent {\bf Gannon's theorem:} Let $T$ be a countable theory and $M$ a (not necessarily countable) model of $T$. Suppose that $p(x)\in S(\cal U)$ is  generically stable over $M$. Then  there is a sequence $(c_i)$ in $M$ such that $\lim tp(c_i/{\cal U})=p$.

\medskip\noindent
 This follows from Theorem~A(1) and the fact that every generically stable type over $M$ is  generically stable over a {\em countable} elementary substructure of $M$. We also give a refinement of Gannon's theorem. In fact, we give a new characterization of generically stable types for countable theories:

\medskip\noindent {\bf Theorem B:}
 	Let $T$ be a countable theory, $M$ a model of $T$, and $p(x)\in S({\cal U})$ a global $M$-invariant type. The following are equivalent:
 	\newline
 	(i) $p$ is generically stable over $M$.
 	\newline
 	(ii) $p$ is definable over $M$, AND there is a sequence $(c_i)$ in $M$ such that $(tp(c_i/{\cal U}):i<\omega)$ $DBSC$-converges to $p$  (as in Definition~\ref{DBSC converge}).
 	\newline
 	Suppose moreover that $T$ has $NSOP$, then (iii) below
 	is also equivalent to (i), (ii) above:
 	\newline
 	(iii)  there is a sequence $(c_i)$ in $M$ such that $(tp(c_i/{\cal U}):i<\omega)$ $DBSC$-converges to $p$.
 
 \medskip
 Notice that, as $DBSC$-convergence is strictly stronger than usual convergence, Theorem B is a clear refinement of Gannon's theorem. Moreover, this result can be lead to an answer to \cite[Question~4.15]{Gannon sequential}.
 
 \medskip 
 Theorems~A and B  allow us to reinforce the main result of \cite{Pillay} on {\em generic} stability in a model. That is,  
 
 \medskip\noindent {\bf Theorem C:}
 Let $T$ be a (countable or uncountable) theory, and $M$ be a model of $T$. The following are equivalent:
 \newline
 (i) Any  type $p\in S_x(M)$ has an extension to a global type $p'\in S_x({\cal U})$ which  is generically stable over $M$.
  \newline 
  (ii) $M$ has no order (as in Definition~\ref{generic model}) AND $M$ is eventually $NIP$.

\medskip
It is worth mentioning that  Gannon's theorem  based on the idea of Simon's lemma, and our results/observations are based on  ideas of both of them. This paper is a kind of companion-piece to \cite{KP} and \cite{K-definable}, although here we are mainly
concerned with model-theoretic proofs of  variants  of    results from \cite{KP}. %In fact, the  results/observations of the present paper are a continuation of the work of many mathematicians, including Ramsey, Morley, Bourgain, Fremlin, Talagrand, Pillay, Simon, and Gannon.
 
 \medskip
 This paper is organized as follows. In Section 2,   we fix some model theoretic conventions.   We will
 also prove  Theorem~A(1). (Cf. Theorem~\ref{Morley sequence}.)
 In Section 3, we will provide all necessary functional analysis notions, and introdce the notion of {\em eventual $NIP$}. We will
 also prove   Theorem~A(2). (Cf. Theorem~\ref{BFT-like}.)
 In Section 4, we will study generically stable types in arbitrary/countable theories.  We will
 also prove    Theorem B and Theorem~C. (Cf. Theorems~\ref{Thm B}, \ref{Pillay-Grothendieck}.) 
At the end paper we conclude some remarks/questions on future generalizations and applications of the results/observations.

\section{Convergent Morley sequences}
The   notation is standard, and a text such as \cite{Simon} will be sufficient background.  We
fix a  first order language $L$, a complete  countable  $L$-theory $T$ (not necessarily $NIP$), and a countable model $M$ of $T$.  The monster model is denoted by $\cal U$ and the space of global types in the variable $x$ is denoted by $S_x({\cal U})$ or $S({\cal U})$. %We let $\phi^*(y, x)= \phi(x, y)$. %Sometimes we write $\phi^*$ instead of $\tilde\phi$.
%We define $p=tp_\phi(a/A)$ as the function$\phi(p,y):A\to\{0,1\}$ by $b\mapsto\phi(a,b)$. This function iscalled a complete $\phi$-types over $A$. The set of all complete$\phi$-types over $A$ is denoted by $S_\phi(A)$. We equip$S_\phi(A)$ with the least topology in which all functions$p\mapsto\phi(p,b)$ (for $b\in A$) are continuous. It is compact and Hausdorff, and is  totally disconnected. Let$X=S_{\phi^*}(A)$ be the space of complete $\phi^*$-types on $A$.  Note that the functions $q\mapsto\phi(a,q)$ (for $a\in A$) are {\em continuous}, and as $\phi$ is fixed we can identify this set of functions with $A$.So, $A$ is a subset of all bounded continuous functions on $X$, denoted by $A\subseteq C(X)$. 

\begin{Convention} In this paper, when we say that $(a_i)\subset {\cal U}$ is a sequence, we mean the usual notion in the sense of analysis. That is, every sequence is indexed by $\omega$. Similarly, we consider Morley sequences  indexed by $\omega$.
\end{Convention}

\begin{Convention} In this paper, a variable $x$ is a tuple of length $n$ (for $n<\omega$).\footnote{Although all arguments are true for infinite variables, to make the proofs readable, we consider finite tuples.} Sometimes we write $\bar x$ or $x_1,\ldots,x_n$ instead of $x$. All types are $n$-type (for $n<\omega$) unless explicitly stated otherwise. Similarly, a sequence  $(a_i)\subset {\cal U}$ is a sequence of tuples of  length $n$ (for $n<\omega$).
\end{Convention}

\begin{Convention} In this paper, when we say that $\phi$ is a formula, we mean a formula over $\emptyset$. 
	 Otherwise, we explicitly say that $\phi$ is an $L(A)$-formula for some set/model $A$. Although the structure of some important definitions and proofs does not depend on the parameter at all.
\end{Convention}

\begin{Definition}
	{\em  Let   $A\subset\cal U$ and $\phi(x_1,\ldots,x_n)\in L(A)$. We say that $\phi(x_1,\ldots,x_n)$ is {\em symmetric} if for any permutation $\sigma$ of $\{1,\ldots,n\}$, $$\models\forall\bar x\big(\phi(x_1,\ldots,x_n)\leftrightarrow\phi(x_{\sigma{(1)}},\ldots,x_{\sigma{(n)}})\big).$$ }
\end{Definition}

For a formula $\phi(x)$ (with or without parameters) and a sequence $(a_i)$ of $x$-tuples in $\cal U$, we write $\lim_{i\to\infty}\phi(a_i)=1$ if there is a natural number $n$ such that ${\cal U}\models\phi(a_i)$ for all $i\geq n$.  If $\lim_{i\to\infty}\neg\phi(a_i)=1$  we write $\lim_{i\to\infty}\phi(a_i)=0$.

For a formula $\phi(x_{i_0},\ldots,x_{i_k})$ and a sequence $(b_i)\in\cal U$, if there exists an $n_\phi$ such that for any $i_k>\cdots>i_0>n_\phi$ we have ${\cal U}\models\phi(b_{i_0},\ldots,b_{i_k})$, we write $\lim_{i_0<\cdots<i_k, \ i_0\to\infty} \phi(b_{i_0},\ldots,b_{i_k})=1$. 

\begin{Definition}
	{\em (i) Let $(b_i)$ be a sequence of elements in $\cal U$ and $A\subset\cal U$  a   set. The   {\em eventual Ehrenfeucht-Mostovski type}\footnote{The $EEM$-type is defined in \cite[Def.~4.3]{Gannon sequential}. It was extracted from the notion of eventual  indiscernible sequence in \cite{S-invariant}.} (abbreviated  {\em $EEM$-type}) of $(b_i)$ over $A$, which is denoted by $EEM((b_i)/A)$, is the following (partial) type in $S_\omega(A)$:
	$$\phi(x_0,\ldots,x_k)\in EEM((b_i)/A)  \iff \lim_{i_0<\cdots<i_k, \ i_0\to\infty} \phi(b_{i_0},\ldots,b_{i_k})=1.$$
%\newline
	(ii) Let $(b_i)$ be a sequence of $\cal U$ and $A\subset\cal U$  a   set. The {\em  symmetric eventual Ehrenfeucht-Mostovski type} (abbreviated {\em $SEEM$-type}) of $(b_i)$ over $A$, which is denoted by $SEEM((b_i)/A)$, is the following partial type in $S_\omega(A)$:
	$$\big\{\phi=\phi(x_0,\ldots,x_n):\phi\in EEM((b_i)/A) \text{ and }\phi \text{ is symmetric} \big\}.$$ 
	Whenever $(b_i)$ is $A$-indiscernible, we sometimes write $SEM((b_i)/A)$ instead of  $SEEM((b_i)/A)$.
\newline
(iii) Let  $p(x)$ be a type in $S_\omega(A)$ (or  $S_\omega({\cal U})$). The  {\em symmetric type of $p$}, denoted by Sym$(p)$, is   the following partial type: $$\big\{\phi(x)\in p: \phi \text{ is symmetric} \big\}.$$ }
\end{Definition}
The  sequence $(b_i)$  is called {\em eventually  indiscernible over $A$} if $EEM((b_i)/A)$ is a {\em complete} type. In this case, for any $L(A)$-formula $\phi(x)$, the limit $\lim_{i\to\infty}\phi(b_i)$ is well-defined.

\begin{Fact}[\cite{Gannon sequential}, Fact 4.2] \label{Gannon fact2}
Let $ (b_i)$ be a sequence of elements  in $\cal U$
  and $A\subset \cal U$ such that $|A| = \aleph_0$.
Then there exists a subsequence $(c_i)$ of  $ (b_i)$  such that $(c_i)$ is eventually
indiscernible over $A$.
\end{Fact}
\begin{proof} 
	A generalization of this observation (for continuous logic) is proved in  Proposition~5.3 of \cite{Gannon sequential}.
\end{proof}

Let $A$ be a set/model and $p(x)$  a global $A$-invariant type. The Morley type (or sequence) of  $p(x)$ is denoted by $p^{(\omega)}$. (Cf. \cite{Simon}, subsection 2.2.1.) The restriction of $p(x)$ to $A$ is denoted by $p|_{A}$.
A realisation $(d_i:i<\omega)$ of $p^{(\omega)}|_A$ is called a Morley sequence of/in $p$ over $A$.

\begin{Lemma} \label{Key lemma}
	Let $p(x)\in S(\cal U)$ be invariant over  $A$, and $I=(d_i)$   a Morley sequence in $p$ over $A$.
	\newline  (i) If there is a sequence $(c_i)\in A$ such that $\lim tp(c_i/AI)=p|_{AI}$ and $(c_i)$ is eventually indiscernible over $AI$, then  $SEEM((c_i)/A)=\text{Sym}(p^{(\omega)}|_A)$.
	\newline  (ii)
	If $p$ is finitely satisfaible in $A$ and $|A|=\aleph_0$, then there is a sequence $(c_i)$ in $A$ such that  $\lim tp(c_i/AI)=p|_{AI}$ and $(c_i)$ is eventually indiscernible over $AI$. Therefore,  $SEEM((c_i)/A)=\text{Sym}(p^{(\omega)}|_A)$.
\end{Lemma}
\begin{proof}
	(i): Suppose that there is a sequence $(c_i)$ such that $\lim tp(c_i/AI)=p|_{AI}$ and $(c_i)$ is eventually indiscernible over $AI$ $(\dagger)$.
	Set $J=(c_i)$.

	We show that    $SEEM((c_i)/A)=SEM((d_i)/A)=\text{Sym}(p^{(\omega)}|_A)$. We remind the reader
	that $SEM((d_i)/A)=\text{Sym}(p^{(\omega)}|_A)$ follows   from the definition of a Morley sequence. The proof is by induction on symmetric formulas. 
	The base case works. Indeed, for any $L(A)$-formula $\phi(x_0)$, $\phi(x_0)\in SEEM((c_i)/A)\iff \lim\phi(c_i)=1\iff \phi(x_0)\in p\iff  \phi(x_0)\in SEM((d_i)/A)$.

	The induction hypothesis is that for any symmetric  formula $\phi(x_0,\ldots,x_{k-1})$ in $L(A)$, $\phi(x_0,\ldots,x_{k-1})\in SEEM((c_i)/A)$ if and only if $\phi(x_0,\ldots,x_{k-1})\in SEM((d_i)/A)$. 
	
	Let $\phi(x_0,\ldots,x_{k})$ be a symmetric $L(A)$-formula  $(\ddagger)$. Clearly, for any $c\in A$, the $L(A)$-formula $\phi(c,x_1,\ldots,x_{k})$ is symmetric. Therefore, by the induction hypothesis,
	
	  $\lim_{i\to\infty}   \phi(c,c_{i+1},\ldots,c_{i+k})= \phi(c,d_1,\ldots,d_{k})$   \ \  $(*)$.
	  
	  \medskip\noindent
	  On the other hand, since $\lim_{n\to\infty} tp(c_n/AI)=p|_{AI}$, we have 
	  
	  $\lim_{n\to\infty}\phi(c_n,d_1,\ldots,d_{k})=\phi(d_{k+1},d_1,\ldots,d_{k})$ \ \  $(**)$.
	  
	 \medskip\noindent
	 To summarize, for large $n$,
	 	\begin{align*}
	\lim_{j\to\infty}\phi(c_j,c_{j+1},\ldots,c_{j+k}) &  \stackrel{(\dagger)}{=} \lim_{n\to\infty}\lim_{i\to\infty}\phi(c_n,c_{i+1},\ldots,c_{i+k})   \\
	 &  \stackrel{(*)}{=}  \lim_{n\to\infty}\phi(c_n,d_1,\ldots,d_{k})  \\
	 & \stackrel{(**)}{=}  \phi(d_{k+1},d_1,\ldots,d_{k}) \\
	 &  \stackrel{(\ddagger)}{=} \phi(d_1,\ldots,d_{k+1}).
	 \end{align*}
	 This means that $\phi(\bar x)\in SEEM(J/A)$ iff $\phi(\bar x)\in SEM(I/A)$.

\noindent\medskip
(ii):	Let $I'$ be a Morley sequence in $p$ over $A$. Since $T$ and $A$ are countable, and $p$ is finitely satisfiable in $A$, there is a sequence $(c_i)$ in $A$ such that $\lim tp(c_i/AI')=p|_{AI'}$. (Notice that the closure of $\{tp(a/AI'):a\in  A\}\subset S_x(AI')$ is second-countable and compact,\footnote{Recall that  a compact Hausdorff space is metrizable if and only if it is second-countable.} and so metrizable. Therefore, there is a   sequence  $(c_i)\in A$ such that $\lim tp(c_i/AI')=p|_{AI'}$.)\footnote{Another argument that is more model-theoretic is given  in the first paragraph of the proof of \cite[Lemma 4.6]{Gannon sequential}.}

 By Fact~\ref{Gannon fact2}, we can assume that $(c_i)$ is eventually indiscernible over $AI$ $(\dagger)$. That is, the type $EEM((c_i)/AI)$ is complete. (Notice that, as $AI$ and $T$ are countable, using Ramsey's theorem and a diagonal argument, there is a subsequence of $(c_i)$ which is eventually indiscernible over $AI$.)

 By (i), $SEEM((c_i)/A)=\text{Sym}(p^{(\omega)}|_A)$.
\end{proof}
\begin{Remark}
  The proof of Lemma~\ref{Key lemma} is essentially the same as  \cite[Lemma~4.5]{Gannon sequential}.  The difference is that we don't need all formulas, but only symmetric ones. Notice that it is not necessary to assume that $A$ is a model. It is worth recalling that  Gannon's result is based on an idea   of Simon \cite[Lemma~2.8]{S-invariant}.
\end{Remark}

Lemma~\ref{Key lemma} above discusses converging of tuples, although  in the rest of paper,  converging of tuples means convergence of types/tuples over the monster model, but not small sets/models:
\begin{Definition} \label{covergence}
{\em We say that a  sequence $(d_i)\in\cal U$ of $x$-tuples  {\em converges} (or {\em is convergent})  if the sequence $(tp(d_i/{\cal U}):i<\omega)$ converges in the logic topology. Equivalently, for any $L({\cal U})$-formula $\phi(x)$, the truth value of  $(\phi(d_i):i<\omega)$ is eventually constant. If $(tp(d_i/{\cal U}):i<\omega)$ converges to a type $p$, then we write $\lim tp(d_i/{\cal U})=p$ or $tp(d_i/{\cal U})\to p$. Notice that $tp(d_i/{\cal U})\to p$ iff for any $L({\cal U})$-formula $\phi(x)$, $$\lim\phi(d_i)=1 \iff \phi(x)\in p.$$ }
\end{Definition}

\begin{Fact}  \label{Rosenthal lemma}
Let $(d_i)$ be a sequence in $\cal U$  of $x$-tuples. Then the following are equivalent:
\newline 
(i) $(d_i)$ has a subsequence with  no convergent subsequence.
 \newline 
(ii) There are a subsequence $(d_i')\subseteq(d_i)$ and a formula $\phi(x,y)$ (with or without parameters) such that for all  (finite) disjoint subsets   $E,F\subseteq\Bbb N$, $$\models\exists y\Big(\bigwedge_{i\in E}\phi(d_i',y)\wedge \bigwedge_{i\notin F}\neg\phi(d_i',y)\Big).$$
  \newline 
 Furthermore, suppose that $(d_i)$ is indiscernible. Then each of (i), (ii) above is also equivalent to (iii) below: 
  \newline 
 (iii) The condition (ii) holds for any subsequence of $(d_i)$. More precisely, there is a formula $\phi(x,y)$ (with or without parameters) such that for any subsequence $(d_i')\subseteq(d_i)$ and   for all  (finite) disjoint subsets   $E,F\subseteq\Bbb N$, $$\models\exists y\Big(\bigwedge_{i\in E}\phi(d_i',y)\wedge \bigwedge_{i\notin F}\neg\phi(d_i',y)\Big).$$ 
\end{Fact}
\begin{proof}
	The direction (i)~$\Rightarrow$~(ii) follows from one of the prettiest result in the Banach space theory due to Rosenthal, Theorem~1 in \cite{Ros}. 
	(See also Lemma~3.12 of \cite{K3} or Appendix B in \cite{K-Banach}.)
	Indeed, as $T$ is countable, we can assume\footnote{If not, using a diagonal argument, we can find a convergent  subsequence of any subsequence.} that there is a  subsequence $(c_i)\subseteq(d_i)$ and a formula $\phi(x,y)$ such that the sequence $(\phi(c_i,y):i<\omega)$ has a subsequence with  no convergent subsequence. Now use Rosenthal's theorem for it.
	 (On the other hand, notice that, as $T$ is countable, every {\bf complete} type can be coded by a function on a suitable space (cf. \cite[Appendix~A]{K-GC}). This leads to an alternative argument.)
\newline	
The direction (ii)~$\Rightarrow$~(iii) follows from indiscernibility.
\newline	
(iii)~$\Rightarrow$~(ii)~$\Rightarrow$~(i) are evident.
\end{proof}
We emphasize that, in Fact~\ref{Rosenthal lemma}, the direction (i)~$\Longrightarrow$~(ii) needs countability of the theory. On the other hand, it is easy to verify that, this fact holds for real-valued functions (or types in continuous logic).

\begin{Theorem} \label{Morley sequence}
	Let $T$ be a countable theory, $M$ a countable model,\footnote{Tanović pointed out to us that it is not necessary to assume that $M$ is a model. Cf.  Remark~\ref{A=any}(iv).}  and $p(x)\in S(\cal U)$ a global type which is finitely satisfiable in $M$. Let $(d_i)$ be a Morley sequence of $p$ over $M$. If $(d_i)$ converges then there is a sequence $(c_i)$ in $M$ such that $tp(c_i/{\cal U})\to p$.
\end{Theorem}
\begin{proof}
 By    Lemma~\ref{Key lemma}, we can assume that  there is a sequence $(c_i)$ in  $M$ such that $tp(c_i/M\cup(d_i))\to p|_{M\cup(d_i)}$ and $SEEM((c_i)/M)=\text{Sym}(p^{(\omega)}|_M)$. We show that  $tp(c_i/{\cal U})\to p$. Let $q$ be an  accumulation point  of $\{tp(c_i/{\cal U}):i\in\omega\}$. Then $q|_{M\cup(d_i)}=p|_{M\cup(d_i)}$. 
 Notice that, as $q$ is finitely satisfiable in $M$ (and so $M$-invariant), the type $q^{(\omega)}$ is well-defined.
 
 	\vspace{4pt}
 \underline{Claim 0}: $p^{(\omega)}|_M=q^{(\omega)}|_M$.
 
 \emph{Proof}: The proof is by induction. The base case is $q|_{M\cup(d_i)}=p|_{M\cup(d_i)}$.
 The induction hypothesis is that $p^{(n+1)}|_M=q^{(n+1)}|_M$. Let $\phi(x_{n+1},x_n,\ldots,x_0)\in L(M)$, and suppose that $p_{x_{n+1}}\otimes p^{(n+1)}_{\bar x} \vdash\phi(x_{n+1},{\bar x})$,  where ${\bar x}=(x_n,\ldots,x_0)$. Since $(d_i)$ is a Morley sequence in $p$ over $M$, $\models\phi(d_{n+1},{\bar d})$ where ${\bar d}=(d_n,\ldots,d_0)$. By definition of Morley sequence, $p_{x_{n+1}}\vdash\phi(x_{n+1},{\bar d})$ and ${\bar d}\models p^{(n+1)}_{\bar x}|_M$. By the hypothesis of induction, ${\bar d}\models q^{(n+1)}_{\bar x}|_M$. By the base case, $q_{x_{n+1}}\vdash\phi(x_{n+1},{\bar d})$, and so by definition, 
 $q_{x_{n+1}}\otimes q^{(n+1)}_{\bar x} \vdash\phi(x_{n+1},{\bar x})$.
 \hfill$\dashv_{\text{claim 0}}$

	\vspace{4pt}
	\underline{Claim 1}: $p=q$.
	
	\emph{Proof}: %If not, using a well-known argument, %in \cite[Fact~1.3]{S-invariant}, there are a formula $\phi(x,y)$, $(a_i)\models p^{(\omega)}|_{M}$ and $b\in\cal U$ such that $\phi(a_i,b)\iff i \text{ is even}$. 
	If not, assume for a contradiction that $p\vdash\phi(x,b)$ and $q\vdash\neg\phi(x,b)$ for some $b\in\cal U$ and formula $\phi(x,y)$ (without parameters). We inductively build  a sequence $(a_i)$ as follows:
	
	 $\bullet$ if $i$ is even, $a_i\models p|_{M\cup\{a_1,\ldots,a_{i-1},b\}}$;
	  
	 $\bullet$ if $i$ is odd, $a_i\models q|_{M\cup\{a_1,\ldots,a_{i-1},b\}}$.
	
	As $p^{(\omega)}|_M=q^{(\omega)}|_M$,  the sequence $(a_i)$  is indiscernible and its type over $M$ is  $p^{(\omega)}|_M$. Moreover,  $\phi(a_i,b)\iff i \text{ is even}$.

	As $(a_i)$ is indiscernible, using the backward direction of \cite[Lemma~2.7]{Simon}, ${\cal U}\models\theta_{n,\phi}(a_1,\ldots,a_n)$\footnote{According to that argument of Claim~0, sometimes it is better to `reverse' the definition of EEM type: replace each $\phi(x_1,\ldots,x_n)$ by
		$\phi(x_n,\ldots,x_1)$. This was suggested to us by  Tanović. Although we still continue the previous arrangement.}  where	
$$\theta_{n,\phi}(x_1,\ldots,x_n)=\forall F\subseteq \{1,\ldots,n\}\exists y_F\Big(\bigwedge_{i\in F}\phi(x_i,y_F)\wedge \bigwedge_{i\notin F}\neg\phi(x_i,y_F)\Big).$$	
(Notice that $\theta_{n,\phi}$ is symmetric, and so $\theta_{n,\phi}\in\text{Sym}(p^{(\omega)}|_M)$.) This means that $\models\theta_{n,\phi}(d_1,\ldots,d_n)$ for all $n$.	Therefore, for any infinite subset $I\subseteq\Bbb N$, the set $$\Sigma_I^\phi(y)=\Big\{ \bigwedge_{i\in I\cap\{1,\ldots,n\}}\phi(d_i,y)\wedge \bigwedge_{i\notin I\cap\{1,\ldots,n\}}\neg\phi(d_i,y): n\in{\Bbb N}\Big\}$$ is a partial type. This means that the sequence $\phi(d_i,y)$ does not converge (and even it has no  convergent subsequence), a contradiction. (Alternatively, as both $(a_i),(d_i)$ are Morley sequence over $M$, there is an automorphism $\sigma\in Aut({\cal U},M)$ which maps $a_i$ to $d_i$. Then $\phi(d_i,\sigma(b))$ converges iff $\phi(a_i,b)$ converges.\footnote{This was suggested to us by Tanović and the referee, independently.})
\hfill$\dashv_{\text{claim 1}}$ 

	\vspace{4pt}
\underline{Claim 2}: The sequence $(tp(c_i/{\cal U}):i<\omega)$ converges. 

	\emph{Proof}: If not,  without loss of generality we can assume that it has no convergent subsequence. (If $(c_i)$ has a convergent subsequence, we can just choose it to be our sequence and explain why it converges to $p$.) As $T$ is countable,  there is a formula $\psi(x,y)$ (with or without parameters) such that, the sequence $(\psi(c_i,y):i<\omega)$ has no convergent subsequence.\footnote{If not, using a diagonal argument one can show the sequence $(tp(c_i/{\cal U}):i<\omega)$ has a convergent subsequence.}  Then, by Fact~\ref{Rosenthal lemma},  for any infinite subset $I\subseteq\Bbb N$, the set
$$\Sigma_I^\psi(y)=\Big\{ \bigwedge_{i\in I\cap\{1,\ldots,n\}}\psi(c_i,y)\wedge \bigwedge_{i\notin I\cap\{1,\ldots,n\}}\neg\psi(c_i,y): n\in{\Bbb N}\Big\}$$ is a partial type. As $\theta_{n,\psi}$ is symmetric, this means that  $\theta_{n,\psi}\in SEEM((c_i)/M)$, and so $\models\theta_{n,\psi}(d_1,\ldots,d_n)$ for all $n$. Equivalently, $(\psi(d_i,y):i<\omega)$ does not converge, a contradiction. \hfill$\dashv_{\text{claim 2}}$

\medskip
Since $(tp(c_i)/{\cal U}:i\in\omega)$ converges to say a type $r$,  this type is in the topological closure of $\{tp(c_i/{\cal U}):i\in\omega\}$. Hence by the first portion of the argument, $r=p$.
\end{proof}

\begin{Remark} \label{A=any}
	(i) 	Let $T$ be a (countable or uncountable) theory, $M$ a model,  and $p(x)\in S(\cal U)$ a global type.	The argument of Claim~1 in the proof of Theorem~\ref{Morley sequence} shows that if there exists a Morley sequence of  $p$ which is convergent, then {\bf any}  Morley sequence of $p$ is convergent.
		\newline
		(ii) Let $T$ be a (countable or uncountable) theory, $M$ a model,  and $p(x),q(x)$ two global $M$-invariant types. If the Morley sequence of $p$ is convergent and $p^{(\omega)}|_M=q^{(\omega)}|_M$, then the argument of  Claim~1 in the proof of Theorem~\ref{Morley sequence} shows that $p=q$.\footnote{Notice that we can assume $p^{(\omega)}|_\emptyset=q^{(\omega)}|_\emptyset$ because the formula $\phi(x,y)$ in Claim~1 has no parameters.} As $T$ is arbitrary, this is a generalization of Proposition~2.36 of \cite{Simon}. (See also Lemma~2.5 of \cite{HP}.)
		\newline
		(iii) There is a converse to Theorem~\ref{Morley sequence}:
		Let $T$ be a countable theory, $M$ a countable model, and $p(x)\in S(\cal U)$ a global type which is finitely satisfiable in $M$.
		If there is a sequence $(c_i)$ in $M$ such that $tp(c_i/{\cal U})$ $DBSC$-converges to $p$ (as in Definition~\ref{DBSC converge}), then some/any Morley sequence of $p$ (over $M$) is convergent. (See the argument of the direction (ii)~$\Longrightarrow$~(iii) of Theorem~\ref{BFT-like}.)
		\newline
		(iv) (Tanović) {Proof of Claim~2 in  Theorem~\ref{Morley sequence}:} If not, there are formulas $\psi(x,y)$ (without parameters) and $b\in\cal U$ such that both sets  $C_1=\{c_i:\models\psi(c_i,b)\}$ and $C_2=\{c_i:\models\neg\psi(c_i,b)\}$ are infinite. Let $p_1,p_2$ be  accumulation points  of $\{tp(c_i/{\cal U}):c_i\in C_1\}$ and $\{tp(c_i/{\cal U}):c_i\in C_2\}$, respectively. Notice that $p_1|_{(d_i)}=p_2|_{(d_i)}$.
		 We inductively build  a sequence $(a_i)$ as follows:
		
		$\bullet$ if $i$ is even, $a_i\models p_1|_{\{a_1,\ldots,a_{i-1}\}}$;
		
		$\bullet$ if $i$ is odd, $a_i\models p_2|_{\{a_1,\ldots,a_{i-1}\}}$.
		
		Similar to Claim~0, we have $p_1^{(\omega)}|_{(d_i)}=p_2^{(\omega)}|_{(d_i)}$, and similar to Claim~1, as the Morley sequence $(d_i)$ is convergent, we have $p_1=p_2$.  
		This is a contradiction, as $p_1\vdash\psi(x,b)$ and $p_1\vdash\neg\psi(x,b)$.
	 (In fact, it is not necessary to assume that $M$ is a model, and we can assume that $M=(c_i)$. However, countability remains a key assumption.)
\end{Remark}
Although with Remark~\ref{A=any}(iv),   we don't need Fact~\ref{Rosenthal lemma}, but for better intuition and providing basic concepts in the rest of the article, the approach of Fact~\ref{Rosenthal lemma} is useful (cf. Definition~\ref{eventual NIP}  below).

\begin{Corollary}[\cite{Gannon sequential},  Theorem~4.8.] \label{G-refine}
Let $T$ be a countable theory,  $p(x)\in S(\cal U)$ and $N$ a (not necessarily countable)  model. If $p$ is generically stable over $N$, then there is a sequence $(c_i)\in N$ such that $tp(c_i/{\cal U})\to p$.
\end{Corollary}
\begin{proof}
As $T$ is countable, there is a countable elementary substructure $M$ of $N$ such that $p$ is generically stable  over  $M$, and so $p$ is finitely satisfiable in $M$ and every Morley sequence of $p$ is convergent. (See also Fact~\ref{Gannon fact} below.)
Then, by Theorem~\ref{Morley sequence}, there is a sequence $(c_i)\in M$ such that $tp(c_i/{\cal U})\to p$.
\end{proof}

Notice that in the proof of Theorem~\ref{Morley sequence}, for any formula $\psi$ there is a natural number $n$ such that $\theta_{n,\psi}\notin SEEM((c_i)/M)$. This is equivalent to a stronger version of convergence that was studied in \cite{K-Baire} and we will recall it in the next section.
 This implies that our result is strictly stronger than Gannon's theorem. Cf. Theorem~\ref{Thm B}, the direction (i)~$\Rightarrow$~(ii). This is also related to Question~4.15 of \cite{Gannon sequential}.

\section{Eventual $NIP$}
In this section, we want to give a characterization of convergent Morley sequences over countable models. First we introduce the following notions.
\begin{Definition} \label{eventual NIP}  {\em Let $T$ be a theory, $M$ a model of it and $\phi(x,y)$ a formula (with or without parameters).%\footnote{We do not know whether the following terminology makes sense or not.}
\newline	
(i) We say that  $\phi(x,y)$ is  {\em eventually $NIP$ in $M$} if for any infinite sequence $(c_i)\in M$ there are a subsequence $(a_i)\subseteq(c_i)$, a natural number $n=n_{(a_i)}$ and  subset $E\subseteq\{1,\ldots,n\}$ such that for any $i_1<\cdots<i_n<\omega$, ${\cal U}\models\psi_\phi(a_{i_1},\ldots,a_{i_n})$  where
$$\psi_\phi(x_1,\ldots,x_n)=\neg\exists y\Big(\bigwedge_{i\in E}\phi(x_i,y)\wedge \bigwedge_{i\in\{1,\ldots,n\}\setminus  E}\neg\phi(x_i,y)\Big).$$
\newline
(ii) We say that  $M$ is  {\em  eventually  $NIP$}  if  {\bf every} formula  is    eventually $NIP$ in $M$. }
\end{Definition}
Note that (i) states that we have a
special {\bf pattern} that never exists. This intuition helps to better understand the notion and how to use it further. In the following, we will explain it better:
\begin{Remark}
	(i) The subsequence $(a_i)$ in the above is convergent for $\phi$.\footnote{An analysis of   $\psi_\phi$ shows that the alternation number $n_{(a_i)}$ of $(\phi(a_i,y):i<\omega)$ is finite. (See also Remark~\ref{explain DBSC}(ii) below.) Notice that $n_{(a_i)}$ depends on both the formula  and the sequence; not just on formula. This is a  `wider’
		notion of alternation number.  Cf. \cite{K-Baire}, the paragraph before Remark~2.11.} That is, $(\phi(a_i,b):i<\omega)$ converges for any $b\in\cal U$.  Moreover, $\psi_\phi(\bar x)$ is in  
	the  Ehrenfeucht-Mostovski type $EM((a_i))$ of $(a_i)$.
	\newline
	(ii) In some sense, the notion  of eventual  $NIP$ is not new. In fact, a theory $T$ is $NIP$ (in the usual sense) iff the  monster model of $T$ is eventually $NIP$ iff every model of $T$  is  eventually $NIP$ iff  some model $M$ of $T$ in which all types over the	empty set in countably many variables are realized  is   eventually $NIP$. Cf. Proposition~2.14 in \cite{K-Baire}.
	\newline
	(iii) The notion of  `eventual $NIP$' is strictly stronger than the notion of `$NIP$ in a model' in \cite{KP}, and strictly weaker than 	
	the notion of `uniform $NIP$ in a model' in \cite{K-Dependent}.\footnote{The directions are consequences of definitions. Although, we strongly believe that the strictness holds, but we have not found clear examples yet. Cf. Remark~\ref{Q1} and item (1) in the ``Concluding remarks/questions" below.}
\end{Remark}

 Let $X$ be a topological space and $f:X\to[0,1]$ be a function. Recall from \cite{K-Baire}  that $f$ is called
 a difference of bounded semi-continuous functions (short $DBSC$)
 if there exist bounded semi-continuous functions $F_1$ and $F_2$ on $X$ with $f=F_1-F_2$.  It is a well-known fact that, in general, the class of $DBSC$ functions is a proper subclass of all Baire~1 function. (Cf. \cite{K-Baire}, Section~2.)
 
 We let $\phi^*(y, x)= \phi(x, y)$. %Sometimes we write $\phi^*$ instead of $\tilde\phi$.
Let $q=tp_{\phi^*}(b/M)$ be the function $\phi^*(q,x):M\to\{0,1\}$ defined by $a\mapsto\phi^*(b,a)$. This function is called a complete $\phi^*$-types over $M$. The set of all complete $\phi^*$-types over $M$ is denoted by $S_{\phi^*}(M)$. We equip $S_{\phi^*}(M)$ with the least topology in which all functions $q\mapsto\phi^*(q,a)$ (for $a\in M$) are continuous. It is compact and Hausdorff, and is  totally disconnected.   
 
 \begin{Definition} \label{DBSC definable}
 	{\em Let $p(x)$ be a global type which is  finitely satisfiable in $M$.
 	
 	\medskip\noindent (1) Suppose that  $\phi(x,y)$ is a formula, and $p_\phi$ is the restriction of $p$ to $\phi$-formulas. Define $f_p^\phi:S_{\phi^*}(M)\to\{0,1\}$    by $f_p^\phi(q)=1$ iff $\phi(x,b)\in p$ for some/any $b\models q$. We say that:
 	
 	(i) $p_\phi$ is {\em definable over $M$} if $f_p^\phi$ is continuous,
 	
 	(ii) $p_\phi$ is {\em $DBSC$ definable over $M$} if $f_p^\phi$ is  $DBSC$,
 	
 	(iii) $p_\phi$ is {\em Baire~1 definable over $M$} if $f_p^\phi$ is Baire~1.
 	
 	\medskip\noindent	(2) $p(x)$  is called ($DBSC$ or {\em Baire~1}) {\em definable over $M$} iff for any formula $\phi(x,y)$ the type $p_\phi$ is ($DBSC$ or Baire~1) definable over $M$, respectively. }
 \end{Definition}
Notice that (i)~$\Rightarrow$~(ii)~$\Rightarrow$~(iii) but in general (i)~$\nLeftarrow$~(ii)~$\nLeftarrow$~(iii). (Cf. \cite{K-Baire}.)

\begin{Definition}  \label{DBSC converge}
{\em	Let $(a_i)\in\cal U$ be a sequence. We say that   $(a_i)$ is {\em $DBSC$-convergent} (or  {\em $DBSC$-converges}) if for any formula $\phi(x,y)$ there is a natural number $N=N_\phi$ such that for any $b\in\cal U$, $$\sum_{i=1}^\infty|\phi(a_i,b)-\phi(a_{i+1},b)|\leq N.$$  }
\end{Definition}

In the following we explain the above notions and their relationship.
 \begin{Remark} \label{explain DBSC}
 	(i) Notice that this notion is equivalent to having finite alternation number. Although, this number depends on both the
 	formula and the sequence; not just on formulas.  
 	\newline
 	(ii) Assuming that $(a_i)$ is eventually indiscernible over the empty set  $\emptyset$. A sequence $(a_i)$ is $DBSC$-convergent iff for any formula $\phi(x,y)$ (without parameters)  there is a formula $\psi_\phi(x_1,\ldots,x_n)$, as be in  Definition~\ref{eventual NIP}, such that $\models\psi_\phi(a_{i_1},\ldots,a_{i_n})$ for any $i_1<\cdots<i_n<\omega$  iff for any formula $\phi(x,y)$ (without parameters) there is a natural number $n$ such that $\theta_{n,\phi}\notin SEEM((a_i)/\emptyset)$.\footnote{$\theta_{n,\phi}$ was defined in the proof of Theorem~\ref{Morley sequence}.}
 	\newline
 	(iii) Suppose that $(a_i)$ is $DBSC$-convergent.
 	  Then the  sequence $tp(a_i/{\cal U})$ converges. Moreover, for any formula $\phi(x,y)$ the sequence $(tp_\phi(a_i/{\cal U}):i<\omega)$ converges to a type $p_\phi$ which is  $DBSC$ definable (over any model $M \supseteq(a_i)$). 
 	  \newline
 	  (iv) Whenever $M$ is countable, $DBSC$-definability and {\em strong Borel definability} (in the sense of \cite{HP}) are the same.
 \end{Remark}
 \begin{proof}
 	(i) was first observed in \cite{K-Baire}. (Cf.  the paragraph before Remark~2.11 in there.)
 	 \newline
(ii) and (iii) follows form Lemma~2.8 of \cite{K-Baire}. For the last part of (ii), note that $\models\psi_\phi(a_{i_1},\ldots,a_{i_n})$  for any $i_1<\cdots<i_n<\omega$, clearly implies that $\theta_{n,\phi}\notin SEEM((a_i)/\emptyset)$. For the converse, suppose that there are natural numbers $n, N$ such that for any $N<i_1<\cdots<i_n$,  $\theta_{n,\phi}\notin  SEEM((a_i)/\emptyset)$. Then, we replace $(a_i)$ by  $(a_{N+1},a_{N+2},\ldots)$ and use Ramsey's theorem, if necessary.
\newline (iv) was first mentioned in \cite[Remark~2.15]{K-Baire} and studied in \cite{K-GC}.
 \end{proof}

We are ready to give a characterization of convergent Morley sequences over countable models in the terms of eventual $NIP$. %and angelicity of the space of finitely satisfiable types.  
 \begin{Theorem}  \label{BFT-like}
Let $T$ be a countable theory and $M$ a  countable model  of  $T$. Then the following are equivalent:
\newline
(i) $M$ is eventually $NIP$.
\newline
(ii) For any $p(x)\in S(\cal U)$ which  is finitely  satisfiable in $M$, there is a sequence $(c_i:i<\omega)\in M$ such that the sequence $(tp(c_i/{\cal U}):i<\omega)$ $DBSC$-converges to  $p$.
\newline
(iii) For any $p(x)\in S(\cal U)$ which  is  finitely  satisfiable  in $M$, there is a Morley sequence $(d_i:i<\omega)$ of $p$ over $M$ such that $(tp(d_i/{\cal U}):i<\omega)$ converges.
 \end{Theorem}
\begin{proof}
	(iii)~$\Rightarrow$~(i): Let $\phi(x,y)$ be a formula, and $(c_i)$ a sequence in $M$. Let $p$ be an  accumulation point  of $\{tp(c_i/{\cal U}):i\in\omega\}$. (Therefore, $p$ is finitely satisfiable in $M$.) Let $I=(d_i)$ be a Morley sequence of $p$ over $M$. By (iii), $(tp(d_i/{\cal U}):i<\omega)$ converges. 
	
		\vspace{4pt}
	\underline{Claim}:
	There is a subsequence  $(a_i)$ of $(c_i)$  such that $\lim tp(a_i/MI)=p|_{MI}$.\footnote{Compare the argument of Lemma~\ref{Key lemma}(ii) above. An alternative (and  more model-theoretic) argument can be given that is similar to the one Gannon gave in the first paragraph of the proof of \cite[Lemma 4.6]{Gannon sequential}.}  
	
	\emph{Proof}: The closure of $\{tp(c_i/MI):i<\omega\}\subset S_x(MI)$ is second-countable and compact, and so metrizable. Therefore, there is a   sequence  $(a_i)\in\{c_i:i<\omega\}$ such that $\lim tp(a_i/MI)=p|_{MI}$. We can assume that $(a_i)$ is a subsequence of $(c_i)$. (If not, consider a subsequence of $(a_i)$ which is a subsequence of $(c_i)$.)
	 %This is possible using a diagonal argument. Indeed, let $\phi_1(x),\phi_2(x),\ldots$ be an  enumeration of $L(MI)$-formulas. As $S_y(MI)$ is separable (i.e. it has a countable dense subset), the sequence $\big(\phi_1(c_i,y):S_y(MI)\to\{0,1\}~|~i<\omega\big)$ has a convergent subsequence, denoted by $\big(\phi_1(c_i^1,y)~|~i<\omega\big)$. Similarly, $\big(\phi_2(c_i^1,y)~|~i<\omega\big)$ has a convergent subsequence $\big(\phi_2(c_i^2,y)~|~i<\omega\big)$, and so on. Set $a_i=c_i^i$ be the diagonal of these sequences. Clearly, by the construction,  $\lim tp(a_i/MI)$ is well-defined, and as $p$ is an accumulation point  of $\{tp(c_i/{\cal U}):i\in\omega\}$,  $\lim tp(a_i/MI)=p|_{MI}$.  
	\hfill$\dashv_{\text{claim}}$
	
 \medskip
 By Fact~\ref{Gannon fact2}, we can assume that $(a_i)$ is eventually indiscernible. Now, by Lemma~\ref{Key lemma}, $SEEM((a_i)/M)=SEM((d_i)/M)$.
	 By Remark~\ref{explain DBSC}, as $(d_i)$ converges, this means that the condition (i) of Definition~\ref{eventual NIP} holds for $(c_i)$ and  $\phi(x,y)$. (Cf. Remark~\ref{A=any}(iv).)
	\newline
	(ii)~$\Rightarrow$~(iii):  Let  $p(x)\in S(\cal U)$ be finitely satisfied in $M$, and $(c_i:i<\omega)\in M$ such that the sequence $(tp(c_i/{\cal U}):i<\omega)$ $DBSC$-converges to  $p$.  Recall that some Morley sequence of $p$ over $M$ is convergent if and  only if {\bf every}  Morley sequence of $p$ over $M$ is convergent.
	Let $I=(d_i)$ be a Morley sequence in $p$ over $M$. By (ii),  $\lim tp(c_i/MI)=p|_{MI}$, and so by Lemma~\ref{Key lemma}, we have $SEEM((c_i)/M)=SEM((d_i)/M)$.
As $(tp(c_i/{\cal U}):i<\omega)$ is $DBSC$-convergent, by Remark~\ref{explain DBSC},  $(d_i)$ converges.
	\newline
	(i)~$\Rightarrow$~(ii): Let  $p(x)\in S(\cal U)$ be finitely satisfied in $M$. By Lemma~\ref{Key lemma},  there are a sequence $(c_i)$ in  $M$ and a Morley sequence $(d_i)$ of $p$ over $M$ such that $tp(c_i/M\cup(d_i))\to p|_{M\cup(d_i)}$ and $SEEM((c_i)/M)=SEM((d_i)/M)$. By~(i), as $T$ is countable, using a diagonal argument, there is a subsequence $(c_i')\subseteq(c_i)$ such that $(tp(c_i'/{\cal U}):i<\omega)$ $DBSC$-converges. Therefore, using an  argument similar to the proof of   Claim~2 in Theorem~\ref{Morley sequence} (or directly), we can see that   the Morley sequence $(d_i)$ is convergent. By an argument similar to Theorem~\ref{Morley sequence},  $(tp(c_i'/{\cal U}):i<\omega)$ converges to $p$. (Cf.  Remark~\ref{A=any}(iv).)
\end{proof}

\begin{Remark} \label{Q1}
	Let $T$ be a countable theory and $M$ a countable model of it. Suppose that any $p(x)\in S(\cal U)$ which  is finitely satisfied in $M$  is $DBSC$~definable over $M$. In this case, using the BFT theorem, it is easy to show that for any such type $p(x)$
  there is a sequence $(c_i)\in M$ such that $\lim tp(c_i/{\cal U})=p$. Notice that  there is no reason that  $(tp(c_i/{\cal U}):i<\omega)$ $DBSC$-converges to $p$. A question arises. With the above assumptions, for any $p(x)\in S(\cal U)$ which  is finitely satisfied in $M$,  is there any sequence $(c_i)\in M$ such that $(tp(c_i/{\cal U}):i<\omega)$ $DBSC$-converges to $p$? We believe that the answer is negative, although we have not found a counterexample yet.
\end{Remark}	
	
\subsection*{An application  to definable groups}
To finish this section, we give an example where the notion of eventual $NIP$  is used to deduce results about definable groups.

\begin{Lemma}  \label{group lemma} 
	 Let $G$ be a definable group. Let $p,q$ be invariant types concentrating on $G$  such that both $p_x\otimes q_y$ and $q_y\otimes p_x$ imply $x\cdot y= y\cdot x$.\footnote{Recall that a type $p(x)$ concentrates on a group $G$ if  $p\vdash x\in G$.}	If some/any Morley sequence of $p$ converges, then $a\cdot b= b\cdot a$ for any $a\models p$ and $b\models q$.
\end{Lemma}		
\begin{proof} The proof is an adaptation of \cite[Lemma~2.26]{Simon}.
By compactness, there is a small model $M$ such that $p,q$ are $M$-invariant and for any $(a,b)$ realizing one of $(p\otimes q)|_M$ or $(q\otimes p)|_M$ we have $a\cdot b= b\cdot a$.

We claim that there is {\bf no}  infinite sequence $(a_n b_n:n<\omega)$ such that $a_n\models p|_{Ma_{<n}b_{<n}}$, $b_n\models q|_{Ma_{<n}b_{<n}}$  and $a_n\cdot b_n\neq b_n\cdot a_n$. If not, by hypothesis  $a_n\cdot b_m= b_m\cdot a_n$ for $n\neq m$. For any $I\subset\omega$ finite, define $b_I=\prod_{n\in I}b_n$. Therefore, $a_n\cdot b_I= b_I\cdot a_n$ if and only if $n\notin I$.  This means that the sequence $(\phi(a_n,y):N<\omega)$ does not converges where $\phi(a_n,y):=a_n\cdot y\neq y\cdot a_n$. As $(a_n:n<\omega)$ is a Morley sequence of $p$, this contradicts the assumption.

Therefore, by the above claim, there is some $n$ such that any sequence with the  above construction has the length  smaller than $n$.
Let $p_0=p|_{Ma_{<n}b_{<n}}$ and $q_0=q|_{Ma_{<n}b_{<n}}$. Then $p_0(x)\wedge q_0(y)\rightarrow  x\cdot y= y\cdot x$.
\end{proof}	

\begin{Proposition}  \label{abeliean subgroup}
Let $T$ be a countable theory and $G$ a definable group. Assume that there is a countable subset $A\subset G$ such that any two elements of $A$ commute, and $A$ is eventually $NIP$. Then there is a definable abelian subgroup of $G$ containing $A$. 
\end{Proposition}
\begin{proof}
Let $S_A\subset S({\cal U})$ be the set of global $1$-types finitely satisfible in $A$. Notice that, as $A$ is eventually $NIP$, by Theorem~\ref{BFT-like}, the Morley sequence of any type in $S_A$ is convergent. Therefore, for any $p,q\in S_A$, the pair $(p,q)$ satisfies the hypothesis of Lemma~\ref{group lemma} above. The rest is similar to the  argument of Proposition~2.27 of \cite{Simon}. Indeed, by Lemma~\ref{group lemma} and compactness, one can find formulas $\phi(x)$ and $\psi(y)$ such that $\phi(x)\wedge\psi(y)\rightarrow x\cdot y=y\cdot x$  and all types of $S_A$ concentrate on both $\phi(x)$ and $\psi(y)$. Set $H:=C_G(C_G(\phi\wedge\psi))$, where $C_G(X)=\{g\in G:g\cdot x=x\cdot g \text{ for all } x\in X\}$. Then $H$ is a definable abelian subgroup of $G$ containing $A$. 
\end{proof}

\begin{Remark}
(1) Notice that, if any two elements of a set $A$ commute, then $C_G(C_G(A))$ is abelian, but $C_G(A)$ is not automatically abelian  (even when $A$ is a subgroup).\footnote{Let $G$
	be any non-abelian group, and let $e$
	be the identity of the group. Then $C_G(e)=G$ is non-abelian.} In the following, we provide a proof:
\newline\noindent
Notice that, as any two elements of $A$ commute, $A\subseteq  C_G(A)$. Therefore $C_G(A)\supseteq  C_G(C_G(A))$.\footnote{Recall that for any $X\subseteq Y$, $C_G(Y)\supseteq C_G(X)$.} Let $a,b\in C_G(C_G(A))$. Since $b\in C_G(A)$, so   by definition $ab=ba$. As $a,b$ are arbitrary, $C_G(C_G(X))$  is abelian.\footnote{This short statement was suggested to us by Narges Hosseinzadeh.}
 %Set $K=C_G(A)$ and let $g_1,g_2 \in C_G(K)$; then:
%$$\forall h \in K, g_1g_2=g_1(hh^{-1})g_2=(g_1h)(h^{-1}g_2)=(hg_1)(g_2h^{-1})=h(g_1g_2)h^{-1} \  (*)$$
%\newline
%Since $g_1g_2=g_2^{-1}(g_2g_1)g_2$, we have:
%$$\forall h \in K, g_1g_2=h(g_1g_2)h^{-1}=h(g_2^{-1}(g_2g_1)g_2)h^{-1}=(hg_2^{-1})(g_2g_1)(hg_2^{-1})^{-1}  \  (**)$$
%\newline
%Since  any two elements of $A$ commute,    $C_G(K) \le K$\footnote{Notice that $A\le C_G(A)$, and so $C_G(A)\ge C_G(C_G(A))$.}, so there is $\bar h \in K$ such that $ g_2=\bar h$; therefore $(**)$ implies:
%$$g_1g_2=(\bar hg_2^{-1})(g_2g_1)(\bar hg_2^{-1})^{-1}=g_2g_1$$
%\newline
%As $g_1,g_2$ are arbitrary, $C_G(C_G(X))$  is abelian.%\footnote{This argument is borrowed from:  https://math.stackexchange.com/questions/3574095/centralizer-of-centralizer-of-an-element-is-abelian } 
\newline
(2) In Proposition~\ref{abeliean subgroup}, if $A$ is finite, we don't need eventual $NIP$: take $H=C_G(C_G(A))$.
\end{Remark}

\section{Generically stable types}	
Here we want to give  new characterizations of generically stable types for countable theories. The notion of generically stable types in general theories was introduced in \cite{Pillay-Tanovic}. Recall from \cite[Prop~3.2]{CG} that a global type $p$ is {\em generically stable} over a small set $A$ if $p$ is $A$-invariant and for {\bf any} Morley sequence $(a_i:i<\omega)$ of $p$ over $A$, we have $\lim tp(a_i/{\cal U})=p$.

\medskip
Before giving the results let us recall that:
\begin{Fact}[\cite{Gannon sequential}, Fact~2.6] \label{Gannon fact}
 Let $M$ be small set,  and $p(x)\in S({\cal U})$ a global $M$-invariant type.
 \newline (i) If $p$ is generically stable over $M$, then $p$ is definable over and finitely satisfiable in $M$.
  \newline (ii)  If $p$ is generically stable over $M$ and $M_0$-invariant, then $p$ is  generically stable over $M_0$. If $p$ is   definable over and finitely satisfiable in $M$ and $M_0$-invariant, the same holds. 
  \newline (iii) Assuming that $T$ is countable, if $p$ is generically stable over $M$, there exists a {\bf countable} elementary substructure $M_0$ such that $p$ is generically stable over $M_0$. The same holds for definable   and finitely satisfiable case.
\end{Fact}

\begin{Lemma} \label{generical stable}
	Let $T$ be a  (countable or uncountable) theory, $A\subset {\cal U}$,  and $p(x)\in S({\cal U})$ a global $A$-invariant type. Suppose that some/any Morley sequence of $p$ is totally indiscernible, AND some/any Morley sequence of $p$ is convergent. Then $p$ is generically stable.\footnote{This was first announced in Remark~3.3(iii) of \cite{HP}.}
\end{Lemma}
\begin{proof}
	Let $I=(a_i)$ be a Morley sequence of $p$ over $A$. We show that $\lim tp(a_i/{\cal U})=p$. Let $\phi(x,b)\in p$ and $J\models p^{(\omega)}|_{AIb}$. Set $I_n=(a_1,\ldots,a_n)$ for all $n$. Notice that all points of $J$ satisfy $\phi(x,b)$, and $I_n+J$ is a Morley sequence (for all $n$).\footnote{$I_n+J$ is the concatenation of $I_n$ and $J$. It has $I_n$ as initial segment and $J$ as the complementary final segment.} We claim that at most a finite number of points of $I$ satisfy $\neg\phi(x,b)$. If not, for each $k$, there is a natural number $n_k$ such that $\#\{a_i\in I_{n_k}:\models\neg\phi(a_i,b)\}\geq k$. As $I_n+J$ is totally indiscernible (for all $n$), this implies that for each $n$,
	 $\theta_{n,\phi}(x_1,\ldots,x_n)\in tp(J)$ where	
	 $$\theta_{n,\phi}(x_1,\ldots,x_n)=\forall F\subseteq \{1,\ldots,n\}\exists y_F\Big(\bigwedge_{i\in F}\phi(x_i,y_F)\wedge \bigwedge_{i\notin F}\neg\phi(x_i,y_F)\Big).$$ (Recall that $\theta_{n,\phi}$ was introduced in the proof of Theorem~\ref{Morley sequence}. Notice that if $\#\{i: \models\phi(a_i,b)\}=\aleph_0$ then we do not need total  indiscernibility, but only indiscernibility.)	
	 Equivalently, $J$ is not convergent, a contradiction.
\end{proof}	

\begin{Remark} \label{one sequence}
Let $T$ be a (countable or uncountable) theory, $A\subset {\cal U}$,  and $p(x)\in S({\cal U})$ a global $A$-invariant type. The following are equivalent.
\newline
(i) $p$ is generically stable.
\newline
(ii) $p$ is definable over a small model AND there is a Morley sequence $(a_i:i<\omega)$ of $p$ over $A$ such that $\lim tp(a_i/{\cal U})=p$.
\end{Remark}
\begin{proof}
(i)~$\Longrightarrow$~(ii) follows from Fact~\ref{Gannon fact}. (Cf. \cite{Simon}, Theorem~2.29.)

(ii)~$\Longrightarrow$~(i): Suppose that there is a Morley sequence $I=(a_i)$ of $p$ over $A$ such that $\lim I=p$. As $p$ is definable and finitely satisfiable,  some/any Morley sequence of $p$ is totally indiscernible. (Cf. \cite[Corollary~4.11]{K-Dependent} for a proof that any definable and finitely satisfiable type commutes with itself and a generalization  to measures.) Therefore, by Lemma~\ref{generical stable}, $p$ is generically stable. 
\end{proof}

The following theorem gives new characterizations of generically stable types for countable theories. The important ones to note immediately are (ii) and (v).
\begin{Theorem}  \label{Thm B}
	Let $T$ be a countable theory, $M$ a small model of $T$, and $p(x)\in S({\cal U})$ a global $M$-invariant type. The following are equivalent:
	\newline
	(i) $p$ is generically stable over $M$.
	\newline
	(ii) $p$ is definable over a small model, AND there is a sequence $(c_i)$ in $M$ such that $(tp(c_i/{\cal U}):i<\omega)$ $DBSC$-converges to $p$.
	\newline
	(iii) $p$ is definable over  and finitely satisfiable in some small model, AND  there is a convergent Morley sequence of $p$ over $M$.
	\newline
	(iv)  $p$ is definable over a small model, AND there is a Morley sequence $(a_i)$ of $p$ over $M$ such that $\lim tp(a_i/{\cal U})=p$.
	\newline
	 Suppose moreover that $T$ has $NSOP$, then  each of (v), (vi), (vii)  below
	is also equivalent to (i), (ii), (iii), (iv) above:
	\newline
	(v)  There is a sequence $(c_i)$ in $M$ such that $(tp(c_i/{\cal U}):i<\omega)$ $DBSC$-converges to $p$.
	\newline
	(vi)    $p$ is finitely satisfiable in a countable model $M_0\prec M$, AND there is a convergent Morley sequence of $p$ over $M$.
	\newline
	(vii) There is a Morley sequence $(a_i)$ of $p$ over $M$ such that $\lim tp(a_i/{\cal U})=p$.
\end{Theorem}
\begin{proof}
 	(i)~$\Longrightarrow$~(ii): As $T$ is countable, by Fact~\ref{Gannon fact}, we can assume that  $p$ is generically stable over a countable substructure  $M_0 \prec M$. By Corollary~\ref{G-refine}, there is a sequence $(c_i)$ in $M$ such that $(tp(c_i/{\cal U}):i<\omega)$ converges to $p$. Notice that in the proof of  Theorem~\ref{Morley sequence} for any formula $\phi$ there is a natural number   $n$ such that the formula $\theta_{n,\phi}$ does not belong to $SEEM((c_i)/M)$. This means that $(tp(c_i/{\cal U}):i<\omega)$ is  $DBSC$-convergent.

 	(ii)~$\Longrightarrow$~(i): Clearly, $p$ is finitely satisfiable in $M$. As $p$ is definable and finitely satisfiable,  any Morley sequence of $p$ is totally indiscernible. (Cf. \cite[Corollary~4.11]{K-Dependent}.)  Let $(d_i)$ be a Morley sequence of $p$ over $M$. By Fact~\ref{Gannon fact2}, we can assume that $(c_i)$ is eventually indiscernible over $M\cup(d_i)$. By Lemma~\ref{Key lemma}, it is easy to see that $SEEM((c_i)/M)=SEM((d_i)/M)$. Therefore, as $(c_i)$ is $DBSC$-convergent, the Morley sequence $(d_i)$ converges. By Lemma~\ref{generical stable}, $p$ is generically stable.

 		(iii)~$\Longrightarrow$~(i) follows from Lemma~\ref{generical stable} and the fact that the Morley sequences of definable and finitely satisfiable types are totally indiscernible.
 		
 	(i)~$\Longrightarrow$~(iii) follows from the direction (i)~$\Longrightarrow$~(ii) of \cite[Pro.~3.2]{CG}. (Recall that generically stable types are definable and finitely satisfiable.)
 	
 	(iv)~$\iff$~(i) follows from Remark~\ref{one sequence}.
 	
 	 	 \medskip
 	The directions   (ii)~$\Longrightarrow$~(v) and (iii)~$\Longrightarrow$~(vi) and (iv)~$\Longrightarrow$~(vii)  are evident (and hold in any theory).
 	
 	 	 \medskip
 	 {\bf  For the rest of the proof, suppose moreover that $T$ has $NSOP$.}
 	 
 	  Then, (v)~$\Longrightarrow$~(ii) follows from Proposition 2.10 of \cite{K-Baire} and the Eberlein--Grothendieck criterion (\cite[Fact~2.2]{K-Baire}).
 	 Indeed, by the direction (i)~$\Longrightarrow$~(iv) of \cite[Pro.  2.10]{K-Baire}, for any formula $\phi(x,y)$,  there is no infinite sequence $(b_j)$ such that $\phi(c_i,b_j)$ holds iff $i<j$. By Fact~2.2 of \cite{K-Baire}, this means that the limit of $(\phi(c_i,y):i<\omega)$ is a continuous function. Equivalently, $p$ is definable over $M$.
 	  (See also Remark~2.11 of \cite{K-Baire}.)

 	(vi)~$\Longrightarrow$~(iii): Suppose that $p$ is finitely satisfiable in  $M_0\prec M$ with $|M_0|=\aleph_0$. By Theorem~\ref{Morley sequence}, there is a sequence $(c_i)\in M_0$ such that $(tp(c_i/{\cal U}):i<\omega)$ $DBSC$-converges to $p$. By the direction (i)~$\Longrightarrow$~(iv) of \cite[Pro.  2.10]{K-Baire} and \cite[Fact  2.2]{K-Baire}, $p$ is definable over $M_0$. Therefore, (iii) holds.
 	
 	(vii)~$\Longrightarrow$~(iv): As $(a_i)$ is {\bf indiscernible} and convergent, the sequence $(tp(a_i/{\cal U}):i<\omega)$ is  $DBSC$-convergent. This means, by $NSOP$ (i.e. the direction (i)~$\Longrightarrow$~(iv) of \cite[Pro.  2.10]{K-Baire} and \cite[Fact  2.2]{K-Baire}), that $p$ is definable.
\end{proof}	

\begin{Remark}
	(i) It is not hard to give a variant of Theorem~\ref{Thm B} for {\em uncountable} theories. Indeed, we can consider {\em all} countable fragments of the languages, and use the above argument.
	\newline 
	(ii) With the assumption of Theorem~\ref{Thm B}, then   $(*)$  below
	is also equivalent to (i)---(iv) in Theorem~\ref{Thm B}:
	
	\medskip
	$(*)$ For any $B\supset M$, $p$ is the unique global nonforking extension of $p|_B$, AND  there is a convergent Morley sequence of $p$ over $M$.
	
	\medskip\noindent
		The argument is an adaptation of the proof of   \cite[Proposition~3.2]{HP}. See also Proposition~\ref{unique extension}(ii) below.
\end{Remark}

As the referee pointed out to us, the following proposition is not new.\footnote{(i) is Remark~5.18 of \cite{CGH}, and (ii) follows from the fact that generically stable types are stationary (cf.  \cite[Proposition 1(iv)]{Pillay-Tanovic}).} Although for the sake of completeness we give a proof using the above observations.

\begin{Proposition} \label{unique extension}
	Let $T$ be a  (countable or uncountable) theory and $p$ a generically stable type.
	\newline (i) For any invariant type $q$, $p\otimes q=q\otimes p$.
	\newline (ii)  If $p$ is $A$-invariant, then $p$ is the unique $A$-invariant extension of $p|_A$.
\end{Proposition}
\begin{proof}
	(i) follows from the argument of Proposition~2.33 of \cite{Simon} by   replacing  \cite[Lemma~2.28]{Simon} with the argument of Lemma~\ref{generical stable} above. Indeed, suppose for a contradiction that for some formula $\phi(x,y,c)\in L({\cal U})$  (where $c$ is a tuple of elements) we have $p_x\otimes q_y\vdash\phi(x,y,c)$ and $q_y\otimes  p_x\vdash\neg\phi(x,y,c)$. Let $(a_i:i<\omega)\models p^{(\omega)}$, $b\models q|_{{\cal U}a_{<\omega}}$ and $(a_i:\omega\leq i<\omega2)\models p^{(\omega)}|_{{\cal U}a_{<\omega}b}$. Then for $i<\omega$, $\neg\phi(a_i,b,c)$ holds and for $i\geq \omega$, we have $\phi(a_i,b,c)$. (Recall the definition of Morley products  in 2.2.1 of \cite{Simon}.) As $(a_i:i<\omega2)$ is totally indiscernible, similar to the argument of Lemma~\ref{generical stable}, it is easy to verify that for each $n$, $\theta_{n,\phi}(x_1,\ldots,x_n)\in tp((a_i)/\emptyset)$ where
	 $$\theta_{n,\phi}(x_1,\ldots,x_n)=\forall F\subseteq \{1,\ldots,n\}\exists y_F\exists y_c\Big(\bigwedge_{i\in F}\phi(x_i,y_F,y_c)\wedge \bigwedge_{i\notin F}\neg\phi(x_i,y_F,y_c)\Big).$$
	 Equivalently, the sequence $(\phi(a_i, y_F,y_c):i<\omega)$ is not convergent, a contradiction.

		(ii): Let $q$ be any $A$-invariant extension of $p|_A$. 
		
		\vspace{4pt}
		\underline{Claim}: $p^{(\omega)}|_A=q^{(\omega)}|_A$.

		\emph{Proof}: The proof is by induction, and similar to the argument of Proposition~2.35 of \cite{Simon}. 
		 The base case is $p|_A=q|_A$. The induction hypothesis is that  $p^{(n)}|_A=q^{(n)}|_A$. Using (i) above and associativity of Morley products, we have:
		 	\begin{align*}
		q^{(n+1)}_{x_1,\ldots,x_{n+1}}|_A &  = (q_{x_{n+1}}\otimes q^{(n)}_{x_1,\ldots,x_{n}})|_A   \\
		& =  (q_{x_{n+1}}\otimes  p^{(n)}_{x_1,\ldots,x_{n}})|_A  \\
		& \stackrel{(*)}{=} (p^{(n)}_{x_1,\ldots,x_{n}}\otimes q_{x_{n+1}})|_A \\  &  = (p^{(n)}_{x_1,\ldots,x_{n}}\otimes  p_{x_{n+1}})|_A
		\\ &  =   	p^{(n+1)}_{x_1,\ldots,x_{n+1}}|_A.
		\end{align*}
Notice that (i) and associativity of Morley products are used in $(*)$.\footnote{Notice that we can not use Lemma~2.34 of \cite{Simon}, because it is not known whether the products of generically stable types are generically stable or not. Although, the associativity of Morley products and the part (i) of Proposition~\ref{unique extension} are sufficient here.} 	\hfill$\dashv_{\text{claim}}$	

 Therefore, every Morley sequence of $q$ is totally indiscernible AND convergent. By Lemma~\ref{generical stable}, $q$ is generically stable and so $\lim I=q$ for any Morley sequence of $q$. This means that $p=\lim I=q$ for any  $I=p^{(\omega)}|_A=q^{(\omega)}|_A$. (Alternatively, as $p^{(\omega)}|_A=q^{(\omega)}|_A$, one can use Remark~\ref{A=any}(ii) above.)
\end{proof}

Here we want to give a local version of a classical result \cite[Proposition~3.2]{HP}:

\begin{Theorem}
	Let $T$ be a (countable or uncountable) theory,  $M$ be a model of  $T$, and  $p(x)$ a global $M$-invariant type. Suppose that there is an elementary extension $M'\succ M$ containing a Morley sequence of $p$ such that $M'$ is eventually $NIP$. Then the following are equivalent.
	
	(i) $p=\lim tp(a_i/{\cal U})$ for any $(a_i)\models p^{(\omega)}|_M$.
	
	(ii) $p$ is definable over and finitely satisfiable in $M$.
	
	(iii)  $p_x\otimes p_y=p_y\otimes p_x$.
	
	(iv) any Morley sequence of $p$ is totally indiscernible.
\end{Theorem}
\begin{proof}
	(i)~$\Longrightarrow$~(ii)~$\Longrightarrow$~(iii)~$\Longrightarrow$~(iv)  are standard and
 hold in any theory. (Cf. Theorem~2.29 of \cite{Simon}.)

	(iv)~$\Longrightarrow$~(i): Let $J\in M'$ be a Morley sequence of $p$. Since   $M'$ is eventually $NIP$, the sequence $J$ is convergent.  By Lemma~\ref{generical stable}, $p$ is generically stable.
\end{proof}
Notice that the above theorem holds with a weaker assumption, namely every formula has $NIP$ in $M'$. (Cf. \cite{KP}, for definition of $NIP$ in a model.)  This easily follows from indiscernibility of Morley sequences.

\subsection*{Eventually stable models}
The story started from Grothendieck's
double limit characterization of weak relative compactness, Theorem~6 in \cite{Grothendieck}. In \cite{Ben-Groth} Ben Yaacov showed that the ``Fundamental Theorem of Stability" is in fact  a consequence of  Grothendieck's theorem.
Shortly afterwards, Pillay \cite{Pillay} pointed out that  the model-theoretic meaning of the Grothendieck theorem is that the formula $\phi(x,y)$ does not have the order property in $M$ if and only if every complete $\phi$-type $p(x)\in S_\phi(M)$ has an extension to a complete type $p'\in S_\phi({\cal U})$
  which is finitely satisfiable in, and definable over $M$. There, he  called such types `generically stable' and  said: ``We will investigate later to what extent we can deduce the stronger notions of generic stability from not the order property in $M$". Here, using the previous results/observations, we can prove a result similar to  \cite{Pillay} for the stronger notions of generic stability. Maybe  the following result is the end of this story, and of course the beginning of another story.

  \begin{Definition} \label{generic model}
  {\em	Let   $M$ be a model. (i)  We say that $M$   {\em has no order}  if for any formula $\phi(x,y)$ there do not exist $(a_i),(b_i)$ in $M$ for $i<\omega$ such that $M\models\phi(a_i,b_j)$ iff $i\leq j$.  
  \newline	(ii) We say that $M$ is {\em eventually stable} if
  
   (1) $M$  has no order, and
   
    (2) $M$   is eventually  $NIP$ (as in Definition~\ref{eventual NIP}). }
  \end{Definition}

  \begin{Remark}
  	(i) In stable theories, every model is eventually  stable.
  	\newline
  	 (ii)  In $NIP$ theories, every  model which has no order is eventually  stable.
  \end{Remark}
\begin{proof}
	(i): If not, similar to the argument of (i)~$\Rightarrow$~(iii) of \cite[Proposition~2.14]{K-Baire},  we can find a formula $\phi(x,y)$, an indiscernible sequence $(c_i)$, and an element $d$ such that   $\phi(c_i,d)$ holds if and only if $i$ is even. This contradicts $NIP$. 
	
	(ii) Suppose that the theory $T$ is $NIP$ and $M\models T$ has no order. 	Suppose for a contradiction  that $M$ is {\bf not} eventually  $NIP$. 	Similar to (i), we can find a formula $\phi(x,y)$, an indiscernible sequence $(c_i)$  {\em (possibly in
	an elementary extension of $M$)}, and an element $d\in \cal U$  such that   $\phi(c_i,d)$ holds if and only if $i$ is even, a contradiction.
\end{proof}

\begin{Theorem} \label{Pillay-Grothendieck} Let $T$ be a (countable or uncountable) theory, and $M$ be a model of $T$. The following are equivalent:
	\newline
	(i) $M$ is eventually  stable.
		\newline
	(ii) Any  type $p\in S_x(M)$ has an extension to a global type $p'\in S_x({\cal U})$ which  is generically stable over $M$.
\end{Theorem}	
\begin{proof} First, without loss of generality we can assume that $T$ is countable.\footnote{We consider {\em all} countable fragments of the languages.} By Proposition~2.3(c) of \cite{Pillay},  $M$ has no order if and only if any  type $p\in S_x(M)$ has an extension to a global type $p'\in S_x({\cal U})$ which  is finitely satisfiable in, and definable over $M$.\footnote{In fact, we do not need to use Grothendieck's argument. Indeed, assuming eventual $NIP$, as any Morley sequence is controlled by a sequence in the model and vice versa (cf. Theorem~\ref{BFT-like}), we can use the standard fact that a Morley sequence  is totally indiscernible iff it has no order (cf. Theorem~12.37 of \cite{Poizat}).}  %Dose $M$ is generically stable iff for any $\phi(x,y)$ and $(a_i)\in M$, there is an eventual subsequence $(a_i')$ such that for each $k$, the pattern $k$-order does not appear eventually in $(a_i')$? We say that the pattern $k$-order does not appear eventually in $(a_i')$ if for each $k$ there is $n_k$ such that the pattern $(\forall {i,j\leq k} \ \phi(a_i,b_j)\iff i<j)$ dose not appear in $(a_i')_{i>n_k}$.} 
   By Theorems~\ref{BFT-like} and \ref{Thm B} above, any global type $p'\in S_x({\cal U})$ which is finitely satisfiable in, and definable over $M$   is generically stable over $M$  if and only if  $M$  is eventually $NIP$.  This proves the theorem.
\end{proof}

\subsection*{Concluding remarks/questions}
(1) In Example~2.18 of \cite{K-Baire}, we built a  graph $N$ with the following property: (i) there is a sequence $(a_i)\in N$ such that $R(a_i,y)$ converges, and (ii) $R(a_i,y)$ is  not $DBSC$-convergent.  We guess that a {\em modification} of this example leads to a definable type $p$ such that: (i) there is a sequence $(a_i)$ with $\lim tp(a_i/{\cal U})=p$, and (ii) $p$ is not the limit of any $DBSC$-convergent sequence. (For this, one need to remove the axiom schema (1) in  Example~2.18, and to check the above properties.)
Therefore,  by Theorem~\ref{Thm B}, $p$ is not generically stable. 
This approach probably  answer to  Question~4.15 of \cite{Gannon sequential}.

%\medskip\noindent(?) In \cite{Pillay}, Pillay pointed out that for a model $M$, every coheir of types over $M$ is definable (over  $M$) if and only if ``$M$ has no order".He postpones a stronger achievement to the next task. Using the result of the present paper, we have the chance to provide the stronger result:For a model $M$, every coheir of types over $M$ is generically stable (on  $M$) if and only if ``$M$ has no order" AND ``$M$ has eventual $NIP$" (as in Definition~\ref{eventual NIP}).

\medskip\noindent
(2) These results/observations can be generalize to ``continuous logic" \cite{BBHU}. On the other hand, one can generalize Theorem~\ref{BFT-like} for {\em measures} in classical logic. This is a generalization of another result of Gannon \cite[Theorem~5.10]{Gannon sequential}. Recall that measures in classical logic correspond to types in continuous logic. This means that a generalization of Theorem~\ref{BFT-like} to continuous logic leads to a generalization of this theorem for measures in classical logic, and vice versa.

\medskip\noindent
(3) In \cite{K-Banach}, we claimed that in the language of Banach spaces in continuous logic, there is a  Krivine-Maurey type theorem for $NIP$ theories (or even $NIP$ spaces). That is, for any separable $NIP$ space $X$ there exists a spreading model of $X$ containing $c_0$ or $\ell_p$
for some $1\leq p<\infty$. We believe that the results/observations of the present paper are sufficient tools and they lead  to a proof of this conjuncture. For example, notice that $EEM$-types correspond to spreading models in Banach space theory. On the other hand, types of $c_0$ and  $\ell_p$ are symmetric in a strong sense. Finally, the types of $c_0$  or  $\ell_p$ are finitely satisfied in any Banach space, by Krivine's theorem.

\medskip\noindent
(4) In \cite{K-Banach}, we showed that every $\aleph_0$-categotical Banach space contains  $c_0$ or $\ell_p$. What is the translation of this observation into ``classical logic" (if such a translation is essentially possible)? Similar questions can be asked about the Krivine-Maurey theorem (and the claim in (3) whenever a proof of it is given).

\medskip\noindent
We will study them elsewhere. (See for example  \cite{K-generic}, for (2).)

\bigskip\noindent
{\bf Acknowledgements.} I want to thank Predrag Tanović for reading a  version of this article and for his helpful comments (especially because of the argument of Remark~\ref{A=any}(iv)).
 
I would like to thank the Institute for Basic Sciences (IPM), Tehran, Iran. Research partially supported by IPM grant~1400030118.

\end{document}